\documentclass[english,a4paper]{amsart}

\usepackage[T1]{fontenc}
\usepackage{ae}
\usepackage{aecompl}
\usepackage{amssymb}
\usepackage{mathrsfs}

\input{xy}
\xyoption{all}
\SelectTips{cm}{}
\usepackage{babel}
\usepackage[pdftex]{graphicx}
\usepackage[leqno]{amsmath}
\usepackage{amsthm}

\usepackage[svgnames]{xcolor}
\usepackage{tikz}
\usetikzlibrary{shapes,patterns,calc,snakes}

\renewcommand\phi{\varphi}
\renewcommand\epsilon{\varepsilon}
\renewcommand\theta{\vartheta}

\newcommand\mbb{\mathbb}

\newcommand\mcal{\mathscr}

\newcommand\ul{\underline}

\newcommand\sA{\mcal{A}}

\newcommand\sS{\mcal{S}}

\newcommand\R{\mbb{R}}

 \newcommand{\ph}[0]{\varphi}
 
 \newcommand{\ep}[0]{\varepsilon}

 \newcommand{\de}[0]{\delta}

 \newcommand{\la}[0]{\lambda}

\DeclareMathOperator\Sym{Sym}

\DeclareMathOperator\interior{\rm int}
\DeclareMathOperator\relint{\rm relint}

\DeclareMathOperator\pr{pr}

\numberwithin{equation}{section}

\theoremstyle{plain}
\newtheorem{Thm}[equation]{Theorem}
\newtheorem{Prop}[equation]{Proposition}
\newtheorem{Cor}[equation]{Corollary}
\newtheorem{Lemma}[equation]{Lemma}

\newtheorem*{Thm*}{Theorem}
\newtheorem*{Satz*}{Satz}
\newtheorem*{Prop*}{Proposition}
\newtheorem*{Cor*}{Corollary}
\newtheorem*{Lemma*}{Lemma}
\newtheorem*{Hilfssatz*}{Lemma}
\newtheorem*{Sublemma*}{Sublemma}
\newtheorem*{Conjecture*}{Conjecture}

\theoremstyle{definition}
\newtheorem{Def}[equation]{Definition}

\newtheorem{Ex}[equation]{Example}

\newtheorem{LemmaDef}[equation]{Lemma and Definition}

\newtheorem{Rem}[equation]{Remark}
\newtheorem{Rems}[equation]{Remarks}

\newtheorem*{Def*}{Definition}
\newtheorem*{Defs*}{Definitions}
\newtheorem*{Ex*}{Example}
\newtheorem*{Exs*}{Examples}
\newtheorem*{LemmaDef*}{Lemma and Definition}
\newtheorem*{Notation*}{Notation}
\newtheorem*{Problem*}{Problem}
\newtheorem*{Question*}{Question}
\newtheorem*{Rem*}{Remark}
\newtheorem*{Rems*}{Remarks}
\newtheorem*{Warning*}{Warning}

\begin{document}
\title{On semidefinite representations of non-closed sets}
\author{Tim Netzer}
\address{Fachbereich Mathematik, Universit{\"a}t Konstanz, 78457 Konstanz, Germany}
\email{tim.netzer@uni-konstanz.de}
\date{\today}

\begin{abstract} Spectrahedra are sets defined by  \textit{linear matrix inequalities}. Projections of spectrahedra are called \textit{semidefinite representable sets}. Both kinds of sets are  of practical use in polynomial optimization, since they occur as feasible sets in semidefinite programming. There are several recent results on the question which sets are semidefinite representable. So far, all results focus on the case of closed sets.
In this work we develop a new method to prove semidefinite representability of sets which are not closed. For example, the interior of a semidefinite representable set is shown to be semidefinite representable. More general, one can remove faces of a semidefinite representable set and preserve semidefinite representability, as long as the faces are parametrized in a suitable way.
\end{abstract}

\maketitle

\section{Introduction}  A linear matrix polynomial  $\sA$ (of dimension $k$, in $n$ variables)  is a symmetric $k\times k$-matrix whose entries are affine linear polynomials over $\R$, in the variables $\ul X =(X_1,\ldots,X_n)$. Equivalenty, it is a linear polynomial in $\ul X$ with coefficients $A_i$ from $\Sym_{k}(\R)$, the space of real symmetric $k\times k$-matrices: $$\sA(\ul X) = A_0 + X_1\cdot A_1 + \ldots + X_n\cdot A_n.$$ For a linear matrix polynomial $\sA$, the set $$\sS(\sA)=\left\{ x\in\R^n\mid \sA(x)\succeq 0\right\}$$ is called a
\textit{spectrahedron} or an \textit{LMI set}. Here, $\succeq 0$ denotes positive semidefiniteness. A spectrahedron is thus a generalization of a polyhedron, which one would obtain by using a diagonal matrix polynomial $\sA$. By using non-diagonal matrices, one can have infinitely many linear inequalities defining $\sS(\sA)$, an inequality $y^t\sA(\ul X)y\geq 0$ for every $y\in\R^k$. 
One can also see spectrahedra as intersections of the cone of positive semidefinite matrices with an affine linear subspace of $\Sym_{k}(\R)$, where the affine subspace is parametrized by $x_1,\ldots, x_n$ (at least if $A_1,\ldots, A_n$ are linearly independent). So the cone of positive semidefinite symmetric $k\times k$-matrices is the standard model of a spectrahedron. 

Spectrahedra are always convex, semialgebraic and closed, even basic closed semialgebraic, i.e. defined by finitely many simultaneous polynomial inequalities. They are also \textit{rigidly convex}, a condition that was first introduced by Helton and Vinnikov \cite{MR2292953}. The authors show that rigid convexity is also sufficient for a two-dimensional set to be a spectrahedron. Lewis, Parrilo and Ramana   \cite{MR2146191} then observed that this proves the Lax conjecture. The question whether every rigidly convex set is a spectrahedron is open for higher dimensions.  

 Also the facial structure of spectrahedra is well known, see for example  Ramana and Goldman \cite{MR1342934}. The authors show that the faces of a spectrahedron are parametrized by subspaces of $\R^k$, and  that  all faces are exposed; see also Section \ref{convsec} below.

Spectrahedra  are of great importance in polynomial optimization. They occur as sets of feasible solutions in semidefinite optimization problems, which are generalizations of linear optimization problems. There exist efficient numerical algorithms to solve such problems, see Boyd, El Ghaoui, Feron and
              Balakrishnan \cite{MR1342934} and Vandenberghe and Boyd \cite{MR1379041} for more information.

Images of spectrahedra under linear projections are still useful for optimization. They  are of the form $$\left\{x\in\R^n\mid\exists y\in\R^m\ \sA(x,y)\succeq 0\right\} ,$$ for some linear matrix polynomial $\sA$ in $n+m$ variables.  Such sets are called \textit{semidefinite representable sets}, and they have recently gained a lot of attention.  Semidefinite representable sets are always convex and semialgebraic, but \textit{no other} necessary condition is known so far. Helton and Nie  \cite{HeltonNieNecSuffSDP} conjecture that \textit{every} convex semialgebraic set is semidefinite representable. 
So far, the following facts are known: 

(i) Every spectrahedron is semidefinite representable. Projections of semidefinite representable sets are semidefinite representable.

(ii) Finite intersections of semidefinite representable sets are semidefinite representable.

(iii) For certain semialgebraic sets $S$, Lasserre's method from \cite{LasserreConvSets} allows to explicitly construct a semidefinite representation, i.e. a spectrahedron that projects to $S$. The method works for \textit{basic closed} semialgebraic sets, i.e. sets defined by finitely many simultaneous polynomial inequalities, and involves sums of squares representions of linear polynomials. Helton and Nie \cite{HeltonNieSDPrepr} have used this method to prove semidefinite representability under certain curvature conditions on the defining inequalities of a set. However, the Lasserre method can only work if all faces of the convex set are exposed, see Netzer, Plaumann and Schweighofer \cite{NePlSch}. So there are basic closed semialgebraic convex sets for which the method fails.

(iv) The convex hull of a finite union of semidefinite representable sets is again semidefinite representable.  This is Helton and Nie \cite{HeltonNieNecSuffSDP}, see also  \cite{NeSi}. So one can apply the Lasserre method locally, at least for compact convex sets. Helton and Nie  \cite{HeltonNieNecSuffSDP} use this to prove additional curvature results.

These seem to be the most important facts on semidefinite representable sets so far. In particular there is a complete lack of results on the semidefinite representability of non-closed semialgebraic sets. In this work we start examining such sets.  We show that the relative interior of a semidefinite representable set is always semidefinite representable. The main result is then Theorem \ref{mainthm} below. It states the we can remove all faces of a semidefinite representable set, except those that are parametrized by another semidefinite representable set, and again obtain a semidefinite representable set. This result allows to produce many new examples.
We start with some helpful results on convex sets and semidefinite matrices.

\section{Lemmas on convex sets and positive semidefinite matrices}\label{convsec} In this section we state some easy (and probably well known) facts about convex sets and matrices. They will be used in Section \ref{open} below.

\begin{LemmaDef} \label{eins}Let $S\subseteq \R^n$ be convex. The \textit{relative interior}  $\relint(S)$ of $S$ is the subset of $S$ that forms the interior of $S$ in the affine hull of $S$.  So a point $x\in S$ belongs to $\relint(S)$ if and only if for all points $y\in S$ there is some $\ep >0$ such that $x+\ep(x-y)\in S$. If $z\in\relint(S)$ then another point $x\in S$ belongs to $\relint(S)$ if and only if there is some $\ep>0$ such that $x+\ep(x-z)\in S$.
 One has $S\subseteq \overline{\relint(S)}$. \end{LemmaDef} \begin{proof}This is an easy exercise.\end{proof}

\begin{Lemma}\label{dense}  Let $S\subseteq \R^n$ be a convex set and let $T$ be a convex subset of $S$ which is dense in $S$. Then $T$ contains the relative interior $\relint(S)$ of $S$.
\end{Lemma}
\begin{proof} Without loss of generality assume that $S$ and therefore also $T$ has nonempty interior in $\R^n$. Now assume for contradiction that there is some $x\in\interior(S)$ that does not belong to $T$. Then by separation of disjoint convex sets, we find an affine linear polynomial $0\neq\ell\in\R[\ul X]$ with $\ell(x)\leq 0$ and $\ell\geq 0$ on $T$. Since $T$ has nonempty interior there is some $y\in T$ with $\ell(y)>0$. Since $T\subseteq S$ and $x\in \interior(S)$ we find some $\ep>0$ such that $y':=x+\ep(x-y)\in S$. Since $\ell(y')<0$ and $\ell\geq 0$ on $\overline{T} $, this contradicts $S\subseteq \overline{T}$.
\end{proof}

\begin{Cor}\label{intproj} Let $S\subseteq \R^{m}$ be convex and let $\ph\colon\R^m\rightarrow \R^n$ be a linear map. Then $$\ph(\relint(S))=\relint(\ph(S)).$$
\end{Cor}
\begin{proof} The inclusion "$\subseteq$" is clear. For "$\supseteq$" notice that since $\relint(S)$ is convex and dense in $S$, $\ph(\relint(S))$ is a convex and dense subset of $\ph(S)$. So the claim follows from Lemma \ref{dense}.
\end{proof}

\begin{Def} Let $S\subseteq \R^n$ be a convex set. A \textit{face of $S$} is a nonempty convex subset $F\subseteq S$ with the following property: for any $x,y\in S$ and $ \la\in (0,1)$, if  $\la x +(1-\la)y\in F$ then $x,y\in F$. 

A face $F$ of $S$ is \textit{exposed}, if either $F=S$ or there is a supporting  hyperplane $H$ of $S$ in $\R^n$ such that $S\cap H=F$. This is equivalent to the existence of an affine linear polynomial $\ell\in\R[\ul X]$ with $\ell\geq 0$ on $S$ and $S\cap \{\ell=0\} =F$.

 \end{Def}
\begin{Lemma}
For every point $x\in S$ there is a unique face $F_x$ of $S$ that contains $x$ in its relative interior. $F_x$ consist precisely of the points $y\in S$ for which there is some $\ep>0$ such that $x+\ep(x-y)\in S$.
\end{Lemma}
\begin{proof} Again an easy exercise. \end{proof}

If $S\subseteq\R^n$ is a spectrahedron, defined by the $k$-dimensional linear matrix inequality $\sA(\ul X)\succeq 0$, then every face of $S$ is of the form $$F_U=\{x\in S\mid U\subseteq \ker \sA(x)\}$$ for some subspace $U$ of $\R^k,$ and one has $F_x=F_{\ker \sA(x)}$ for all $x\in S$; every face of $S$ is exposed  (see \cite{MR1342934} and also \cite{NePlSch}).

We now turn to matrices. 
 The next Proposition will be crucial for the results in Section \ref{open}.
 
 \begin{Prop} \label{prop}Let $A\in\Sym_k(\R)$ and $B\in \R^{m\times k}$. Let $I_m$ denote the identity matrix of dimension $m$. Then the following are equivalent:
 \begin{itemize}
 \item[(i)] there is some $\lambda\in\R$ such that $\left(\begin{array}{c|c} A& B^t \\\hline B & \lambda \cdot I_m\end{array}\right)\succeq 0$
 \item[(ii)] $A\succeq 0$ and $\ker A\subseteq \ker B$
 \end{itemize}
 \end{Prop}
 \begin{proof} By Theorem 1 in Albert \cite{MR0245582}, (i) is equivalent to the existence of some $\la$ such that $$A\succeq 0,\ B=BA^{\dagger}A, \ \la\cdot I_m -BA^{\dagger}B^t\succeq 0,$$ where $A^{\dagger}$ denotes the Penrose-Moore pseudoinverse matrix to $A$. By Theorem 9.17 in Ahlbrandt and Peterson \cite{MR1423802}, condition $B=BA^{\dagger}A$ is equivalent to $\ker A\subseteq\ker B$. Finally, one can always choose some big enough $\la$ to insure $\la\cdot I_m -BA^{\dagger}B^t\succeq 0$, which proves the Proposition. 
  \end{proof}

\section{Non-closed semidefinite representable sets}\label{open}

All of the existing results on semidefinite representations of sets concern \textit{closed} sets. Our goal in this section is to start examining non-closed sets.

The following easy result states that we can always remove faces of semidefinite representable sets, and still obtain semidefinite representability. It does not use the results from Section \ref{convsec}  yet.

\begin{Prop}\label{faceoff} If $S$ is semidefinite representable and $F$ is a face of $S$, then $F$ and $S\setminus F$ are semidefinite representable.
\end{Prop}
\begin{proof}  First assume that $S$ is a spectrahedron, defined by the linear matrix polynomial $\sA$. Then $F$ is an exposed face of $S$ (by \cite{MR1342934}, Corollary 1), which means that there is an affine linear polynomial $\ell\in\R[\underline{X}]$ such that $\ell\geq 0$ on $S$ and $\{\ell=0\}\cap S=F.$ So we have $$ F=\left\{x\in\R^n\mid \sA(x)\succeq 0 \wedge \ell(x)=0\right\}$$ and $$S\setminus F=\left\{x\in \R^n\mid \sA(x)\succeq 0 \wedge  \exists  \lambda \left(\begin{array}{cc}\lambda  & 1 \\1 & \ell(x)\end{array}\right)\succeq 0 \right\}.$$ This shows that $F$ is even a spectrahedron and $S\setminus F$ is semidefinite representable.

Now let $S$ be semidefinite representable and let $\widetilde{S}\subseteq \R^{n+m}$ be a spectrahedron such that $S$ is the image of $\widetilde{S}$ with respect to the projection $\pr\colon \R^{n+m}\rightarrow \R^n$. Then $\widetilde{F}:=\pr^{-1}(F)\cap \widetilde{S}$ is a face of $\widetilde{S}$. Since $\widetilde{F}$ projects onto $F$  and  $\widetilde{S}\setminus \widetilde{F}$ projects onto $S\setminus F$,  both sets are semidefinite representable.
\end{proof}

For a semidefinite representable set with only finitely many faces, i.e. for a polyhedron, we thus know that its interior is again semidefinite representable. But this result is true in general:
 \begin{Prop}\label{sdpint} If $S$ is semidefinite representable, then $\relint(S)$ is also semidefinite representable.
\end{Prop}

\begin{proof}First assume that $S$ is a spectrahedron, defined by the matrix polynomial  $\sA(\ul X)=A_0 +X_1A_1 +\ldots + X_n A_n.$ Fix a point $z\in \relint(S)$. By Lemma \ref{eins}, $\relint(S)$ has the following description: $$\relint(S)=\left\{ x\in S\mid \exists \ep>0\ x+\ep(x-z)\in S\right\}.$$ For $\ep>0$ we have $\sA(x+\ep(x-z))\succeq 0$ if and only if $\frac{1}{1+\ep}\cdot \sA(x+\ep(x-z))\succeq 0$, and  \begin{align*}  \frac{1}{1+\ep}\cdot \sA(x+\ep(x-z)) &=  \left( \frac{1}{1+\ep}  \right)\cdot A_0 +x_1A_1 +\cdots +x_nA_n \\ &\quad -\left(\frac{\ep}{1+\ep}\right)\cdot \left( z_1A_1+\cdots +z_nA_n\right).\end{align*}  Making the transformation $\de:=\frac{1}{1+\ep}$ and writing $B:= -(z_1A_1+\cdots +z_nA_n)$ we find $$ \relint(S)=\left\{x\in \R^n\mid \exists \delta\in (0,1)\quad  \de A_0 + x_1A_1+\cdots + x_nA_n + (1-\de)B\succeq 0\right\}.$$ Since the condition $\de\in (0,1)$ can be translated into $$\exists\la \   \left(\begin{array}{cc}\la & 1 \\1 & \delta\end{array}\right) \succeq 0 \wedge \left(\begin{array}{cc}\la & 1 \\1 & 1- \delta\end{array}\right)\succeq 0,$$ this is clearly a semidefinite representation of $\relint(S)$.

Now let $S$ be semidefinite representable and suppose $\widetilde{S}\subseteq \R^{n+m}$ is a spectrahedron that projects to $S$. Then $\relint(\widetilde{S})$ projects onto $\relint(S)$, by Corollary \ref{intproj}. Since we already know that $\relint(\widetilde{S})$ is semidefinite representable, this proves the claim.
\end{proof}

\begin{Rem}We also have some quantitative information in this last result. Assume that $S\subseteq\R^n$ is semidefinite representable and $\widetilde{S}\subseteq\R^{n+m}$ is a spectrahedron that projects to $S$. If $\widetilde{S}$  is defined by a $k$-dimensional linear matrix polynomial, then $\relint(S)$ is the image of a spectrahedron in $\R^{n+m+2},$ defined by a linear matrix polynomial of dimension $k+4$.  This is clear from the proof of Proposition \ref{sdpint}. \end{Rem}

\begin{Rem} We could also try to quantify the element $z$ in the proof of Proposition \ref{sdpint}, instead of only using one fixed $z$ from $\relint(S)$. This would allow us to be more sophisticated in removing faces of $S$. However, the approach from the proof doesn't  seem to work then. It relies on the fact that we consider $z$ as a fixed parameter. Otherwise we can not get rid of the product $(1+\ep)x$ by dividing through $1+\ep$. However, we can still prove something better, using a different method. This is our main result, Theorem \ref{mainthm} below.
\end{Rem}

By now we have shown that we can remove \textit{finitely many faces} or \textit{all  faces} of codimension $\geq 1$ from a semidefinite representable set, and obtain a semidefinite representable set. But with the results from the previous section we can prove more. We start with spectrahedra (recall the notations from Section \ref{convsec}):

\begin{Prop}\label{sub}Let $S$ be defined by the $k$-dimensional linear matrix polynomial $\sA(\ul X)$. Then for every subspace $W$ of $\R^k$, the set
 $$\left\{x\in S\mid \ker \sA(x)\subseteq W\right\}=S\setminus \bigcup_{U \nsubseteq W}F_U$$ is semidefinite representable.
\end{Prop}
\begin{proof} Choose an $m\times k$-matrix $B$ with $\ker B=W$. By Proposition \ref{prop} we find $$\left\{x\in S\mid \ker \sA(x)\subseteq W\right\}= \left\{ x\in\R^n\mid \exists \la\ \left(\begin{array}{c|c}\sA(x) & B^t \\\hline B & \la\cdot I_m\end{array}\right)\succeq 0 \right\},$$ which is a semidefinite representation. \end{proof}

\begin{Rem} If $S$ has nonempty interior, then the linear matrix polynomial $\sA(\ul X)$ can be chosen such that $\sA(\ul X)\succ 0$ defines $\interior(S)$, see \cite{MR2292953}. Then $$\interior(S)=\{ x\in S\mid \ker \sA(x)\subseteq \{0\} \}$$ is semidefinite representable by Proposition \ref{sub}. This is another way to prove Proposition \ref{sdpint}. 

\end{Rem}

\begin{Ex} \label{ex1}Let $D_2$ be the unit disk in $\R^2,$ defined by the linear matrix polynomial  $$\sA(X_1,X_2):=\left(\begin{array}{cc}1-X_1 & X_2 \\X_2 & 1+X_1\end{array}\right),$$ as above. The faces of $D_2$ are $D_2$ itself and the points on the boundary of $D_2$.
For $(x_1,x_2)\in D_2$ we have $$\ker \sA(x_1,x_2)=\left\{\begin{array}{ll}\{0\} & \mbox{ if }  x_1^2+x_2^2<1 \\\R\cdot (x_2,x_1-1) & \mbox{ if } x_1^2+x_2^2=1, x_1\neq 1\\ \R\cdot (1,0) & \mbox{ if } (x_1,x_2)=(1,0)\end{array}\right.$$ So one checks that for any one-dimensional subspace $W$ of $\R^2$, the set $$\left\{ (x_1,x_2)\in S\mid \ker \sA(x_1,x_2)\subseteq W\right\}$$ is the open unit disk together with one point on the boundary. Since the convex hull of a finite union of  semidefinite representable sets is again semidefinite representable (by \cite{HeltonNieNecSuffSDP}, Theorem 2.2), we obtain that  the open unit disk together with finitely many points on the boundary is semidefinite representable. By Proposition \ref{faceoff}, also $D_2$ with finitely many points on the boundary removed is semidefinite representable.
\end{Ex}

So Propositions \ref{faceoff}, \ref{sdpint} and \ref{sub} tell us that we can either remove finitely many faces or "almost all" of the faces of a spectrahedron and obtain a semidefinite representable set.  But we would also like to do something in between, for example remove a semi-arc from the boundary of the disk. This leads to our main result:

For a convex set $S$ and $z\in S$ we denote by $\mathcal{F}(z,S)$ the set of all faces of $S$ that contain $z$. In particular always $S\in \mathcal{F}(z,S)$. For a set $T\subseteq S$ we denote by $(T\looparrowleft S)$ the union of the interiors of all faces of $S$ that are touched by $T$, i.e. $$(T\looparrowleft S) :=\bigcup_{z\in T}\ \bigcup_{ F\in \mathcal{F}(z,S)}\relint(F).$$

\begin{Thm} \label{mainthm}Let $T\subseteq S\subseteq \R^n$ be semidefinite representable sets. Then $(T\looparrowleft S)$ is also semidefinite representable.
\end{Thm}
\begin{proof} First assume that $S$ is a spectrahedron. Let $\sA(\ul X)$ be a $k$-dimensional symmetric linear matrix polynomial defining $S$. For any $z\in T$ we have \begin{align*}\bigcup_{F\in \mathcal{F}(z,S) } \relint(F) & = \left\{ x\in S\mid z\in F_x\right\} \\ &= \left\{ x\in \R^n\mid \sA(x)\succeq 0,  \ker \sA(x)\subseteq \ker \sA(z)\right\}. \end{align*} So by Proposition \ref{prop} we have $$(T\looparrowleft S)=\left\{ x\in\R^n\mid \exists z\in T\ \exists \la \ \left(\begin{array}{c|c}\sA(x) & \sA(z) \\\hline \sA(z) & \la\cdot I_k\end{array}\right)\succeq 0\right\},$$ which is a semidefinite representation.

Now let $S$ be semidefinite representable. So there is a spectrahedron  $\widetilde{S}$ in some $\R^{n+m}$ that projects onto $S$ via the projection map $\pr\colon \R^{n+m}\rightarrow \R^n$. Define $$\widetilde{T}:=\pr^{-1}(T)\cap \widetilde{S}=\left\{ (x,y)\in \R^{n+m}\mid (x,y)\in \widetilde{S} , x \in T\right\},$$ which is clearly a semidefinite representable subset of $\widetilde{S}.$ We now know that $(\widetilde{T}\looparrowleft \widetilde{S})$ is semidefinite representable, so we  finish the proof by showing $$ \pr\left((\widetilde{T}\looparrowleft \widetilde{S})\right) = (T \looparrowleft S).$$

For "$\subseteq$" let $(x,y)\in (\widetilde{T}\looparrowleft \widetilde{S})$  be given. We have to show $x\in (T\looparrowleft S)$. There is some $(v,w)\in \widetilde{T}$ and some face $\widetilde{F}\in \mathcal{F}((v,w),\widetilde{S})$ such that $(x,y)\in\relint(\widetilde{F}).$ So there is some $\ep>0$ such that $(x,y)+\ep\left((x,y)-(v,w)\right)\in \widetilde{F}.$ So $x+\ep(x-v)\in\pr(\widetilde{F})\subseteq S.$ This implies $v\in F_x$, so $F_x\in\mathcal{F}(v,S)$ and clearly $x\in\relint(F_x)$. Since $v\in T$ this proves $x\in (T\looparrowleft S)$.

For "$\supseteq$" let $F$ be a face of $S$ that contains some element from $T$. Then $\widetilde{F}:=\pr^{-1}(F) \cap \widetilde{S}$ is a face of $\widetilde{S}$ that contains some element from $\widetilde{T}$. By Corollary \ref{intproj} we find $$\pr\left(\relint(\widetilde{F})\right)= \relint\left( \pr(\widetilde{F})\right)= \relint(F),$$ which proves the desired inclusion.
\end{proof}

\begin{Rems}
\begin{itemize}
\item[(0)] One has $(S\looparrowleft S)=S$ and $(\emptyset\looparrowleft S)=\emptyset$ for any convex set $S$. Clearly $T\subseteq T'\subseteq S$ implies $(T\looparrowleft S)\subseteq (T'\looparrowleft S).$

\item[(i)] For  a point $x\in \relint(S)$ one has $(\{x\}\looparrowleft S)= \relint(S).$ So Theorem \ref{mainthm} generalizes Proposition \ref{sdpint} from above.

\item[(ii)] $(T\looparrowleft S)$ always contains $T$, and also $\relint(S)$ as long as $T\neq \emptyset$. 

\item[(iii)] The semidefinite representation of $(T\looparrowleft S)$ is explicitly given in the proof of Theorem \ref{mainthm}. So one for example checks that it preserves rational coefficients from a semidefinite representation of $T$ and $S$.
\end{itemize}

\end{Rems}

\begin{Ex} Let $D_2$ be the unit disk in $\R^2$. We find that we can remove any arc in the boundary of $D_2$ (and therefore any semialgebraic subset of the boundary) and obtain a semidefinite representable set. This is implied by Theorem \ref{mainthm}. For any arc in the boundary of $D_2$ one simply has to provide a semidefinite representable subset $T$ of $D_2$ that touches the boundary of $D_2$ precisely in the points that do not belong to the given arc. This is always possible, as one easily checks. 
\end{Ex}

\begin{Ex}Consider the following subset $S$ of $\R^2$:  $$S=D_2\cup \left([-1,1]\times [0,1]\right).$$ $S$ is not a spectrahedron, since it is not even basic closed semialgebraic (and has a non-exposed face). But it is semidefinite representable, which for example follows from  Theorem 2.2 in Helton and Nie \cite{HeltonNieNecSuffSDP}. Now consider the subset $T$ of $S$ defined by $$T=\{(x,y)\in S\mid \vert x \vert -1  \leq y \leq 0  \}.$$ Then $(T\looparrowleft S)$ consists of  $\interior(S)$ together with the  point $ (0,-1)$ and the set $\{-1,1\}\times [0,1).$
Since $S$ and $T$ are semidefinite representable, so is $(T\looparrowleft S)$.
\end{Ex}

{\linespread{1}\bibliographystyle{plain} \bibliography{references}}
\end{document}